\documentclass[11pt]{article}
\usepackage[english]{babel}
\usepackage[cp1252]{inputenc}
\usepackage{graphicx}
\usepackage{float}
\usepackage{amsmath, amssymb, amsthm}
\usepackage{indentfirst}

\newtheorem{Teo}{Theorem}[section]
\newtheorem{Def}[Teo]{Definition}
\newtheorem{Prop}[Teo]{Proposition}
\newtheorem{Lema}[Teo]{Lemma}

\newtheorem{Cor}[Teo]{Corollary}

\begin{document}

\title{\textsc{Strongly quasipositive links with braid index 3 have positive Conway polynomial}}

\author{Marithania Silvero \footnote{Partially supported by MTM2010-19355, P09-FQM-5112 and FEDER.}\\
Departamento de Álgebra.
Facultad de Matemáticas. \\
Universidad de Sevilla.
Spain.\\
{\tt marithania@us.es}\\ \\
}

\maketitle


\noindent \textbf{Abstract} \,

Strongly quasipositive links are those links which can be seen as closures of positive braids in terms of band generators. In this paper we give a necessary condition for a link with braid index 3 to be strongly quasipositive, by proving that in that case it has positive Conway polynomial (that is, all its coefficients are non-negative). We also show that this result cannot be extended to a higher number of strands, as we provide a strongly quasipositive braid whose closure has non-positive Conway polynomial.

\bigskip

\vspace{0.2cm}
\mbox{ }


\section{Introduction}

The notion of positivity related to a link has been deeply studied from many different points of view. Maybe the simplest class involving this concept is the family of positive links. An oriented link is said to be positive if it has a positive diagram, that is, a diagram with all crossings being positive (see Figure \ref{posArtinBKL}).

There are also analogous notions of positivity related to braids; they depend on the choice of the presentation of the braid group. Roughly speaking, a braid is positive if there exists a positive word representing it, that is a word with all its letters having positive exponent.

The braid group on $n$ strands, $\mathbb{B}_n$, has a standard presentation due to Artin (\cite{Artin1}, \cite{Artin2}):

$$
\mathbb{B}_n = \left< \sigma_1, \sigma_2, ... , \sigma_{n-1} \left| \begin{array}{cccc}
                                                            \sigma_i\sigma_j\sigma_i = \sigma_j\sigma_i\sigma_j & & &  |i-j| = 1 \\
                                                            \sigma_i\sigma_j = \sigma_j\sigma_i & & &  |i-j| > 1
                                                            \end{array}
                                                    \right.
\right>
$$

Attending to the presentation above, a braid is said to be Artin-positive if it can be represented by a positive standard braid word, that is, a braid word where each Artin generator $\sigma_i$ appears with positive exponent. The closure of an Artin-positive braid is an Artin-positive link.

Artin-positive links are, indeed, positive links. However, the converse is not true. Although many proofs of this fact can be given, in Section 6 we present one that could be interesting for the reader.

\begin{figure}
\centering
\includegraphics[width = 12cm]{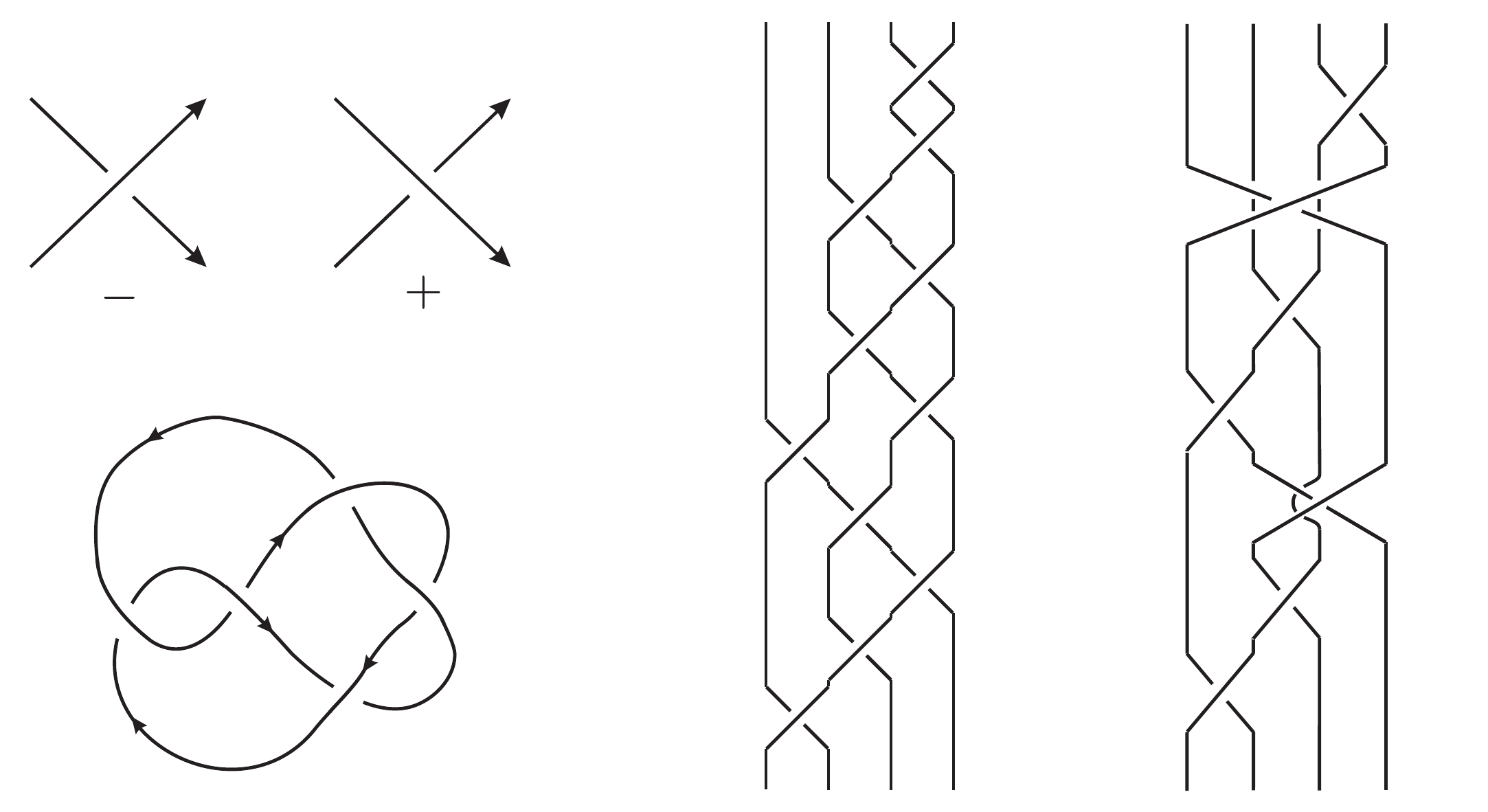}
\caption{\small{In the leftmost part of the figure, you can see the chosen convention of signs and a positive diagram representing the positive knot $5_2$. The Artin-positive braid $\sigma_3\sigma_3\sigma_2\sigma_3\sigma_2\sigma_3\sigma_1\sigma_2\sigma_3\sigma_2\sigma_1$ and the BKL-positive braid $\sigma_{34}\sigma_{14}\sigma_{23}\sigma_{12}\sigma_{24}\sigma_{23}\sigma_{12}$ are also shown.}}
\label{posArtinBKL}
\end{figure}

The braid group $\mathbb{B}_n$ admits another well known presentation due to Birman, Ko and Lee \cite{BirmanKoLee}. The so called BKL generators or \emph{band generators}, $\sigma_{ij}$, are related to Artin-generators by the formula $\sigma_{ij}= (\sigma_{j-2} \ldots \sigma_i)^{-1} \sigma_{j-1} (\sigma_{j-2} \ldots \sigma_{i})$, with $i < j$. They correspond to a positive crossing of strands in positions $i$ and $j$ passing in front of the other strands, as shown in Figure \ref{posArtinBKL}.

$$
\mathbb{B}_n = \left< \sigma_{rs}, \, 1 \leq r < s \leq n \left| \begin{array}{cccc}
                                                            \sigma_{st}\sigma_{qr} = \sigma_{qr}\sigma_{st} & & (t-r)(t-q)(s-r)(s-q) > 0 \\
                                                            \sigma_{st}\sigma_{rs} = \sigma_{rt}\sigma_{st} = \sigma_{rs}\sigma_{rt} & & 1 \leq r < s < t \leq n
                                                            \end{array}
                                                    \right.
\right>
$$

Just as before, a BKL-word having only positive exponents is a BKL-positive word. A braid is BKL-positive if there exists a BKL-positive word representing it, and the closure of a BKL-positive braid is said to be a BKL-positive link. Being more precise:

\begin{Def}
A braid is BKL-positive if it can be expressed by a word written in terms of the generators given by Birman, Ko and Lee, with all letters having positive exponents. The closure of a BKL-positive braid is a BKL-positive link.
\end{Def}

In \cite{RudolphStronglyquasipos} Rudolph introduced these links as the boundaries of what he called \emph{quasipositive surfaces}. A quasipositive surface is basically an orientable surface consisting in a finite number of parallel discs joined by some bands twisted in a positive way (see Figure \ref{quasipossurf}). Each of these bands corresponds to a BKL-generator.

\begin{figure}
\centering
\includegraphics[width = 7.5cm]{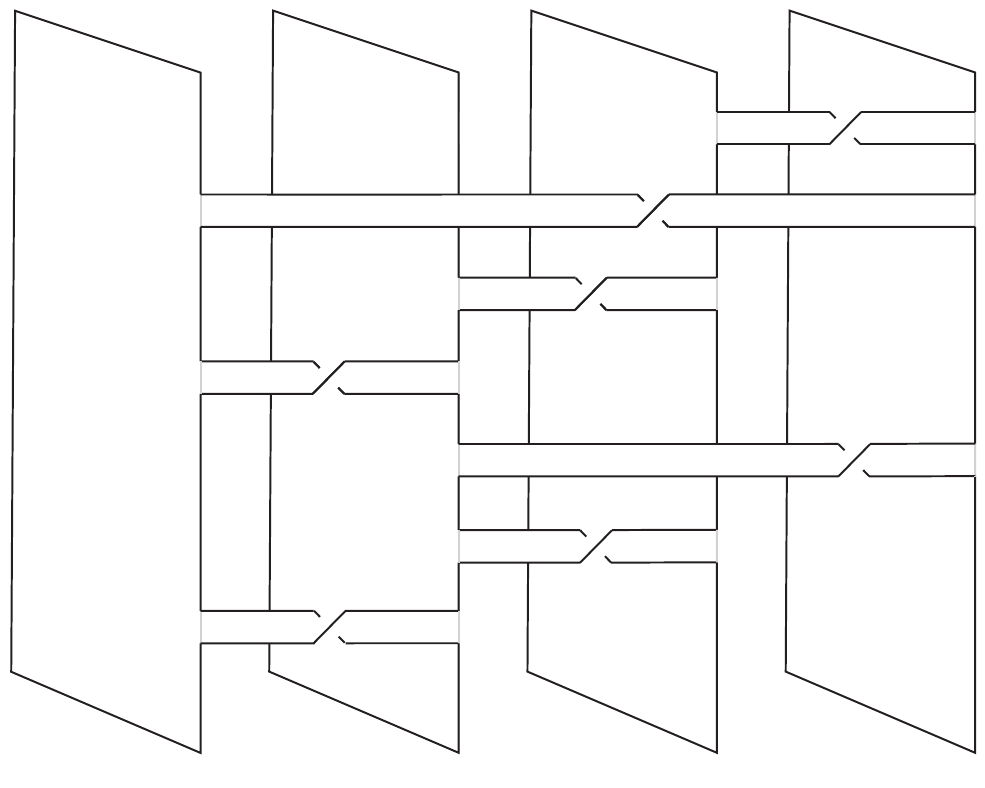}
\caption{\small{The quasipositive surface associated to the closed BKL-positive braid in Figure 1. Its boundary is a strongly quasipositive link (or BKL-positive link).}}
\label{quasipossurf}
\end{figure}

Hence the families of strongly quasipositive links and BKL-positive links are equal. The first name reminds that the link is boundary of a nice kind of surface; the second one brings to our mind the algebraic presentation of the braid group.

In \cite{Rudolph2} Rudolph proved that positive links are strongly quasipositive, which is not obvious. The converse is not true. In fact, Baader \cite{Baader} showed that a link is positive if and only if it is strongly quasipositive and homogeneous; in particular, the link $L9n18\{1\}$ is strongly quasipositive but not positive \cite{MarithaConj}.

The Alexander-Conway polynomial (or just Conway polynomial) was re-discovered in 1969 by J. Conway, as a normalized version of the Alexander polynomial. Precisely, we will refer to the polynomial $\nabla(L) \in \mathbb{Z}[z]$ given by the skein relation $\nabla(L_+) - \nabla(L_-) = z \nabla(L_0)$ with normalization $\nabla(\bigcirc) = 1$, where $L_+, L_-$ and $L_0$ have diagrams that only differ in a small neighborhood, as shown in Figure \ref{D+-0}.

\begin{figure}[H]
\centering
\includegraphics[width = 6cm]{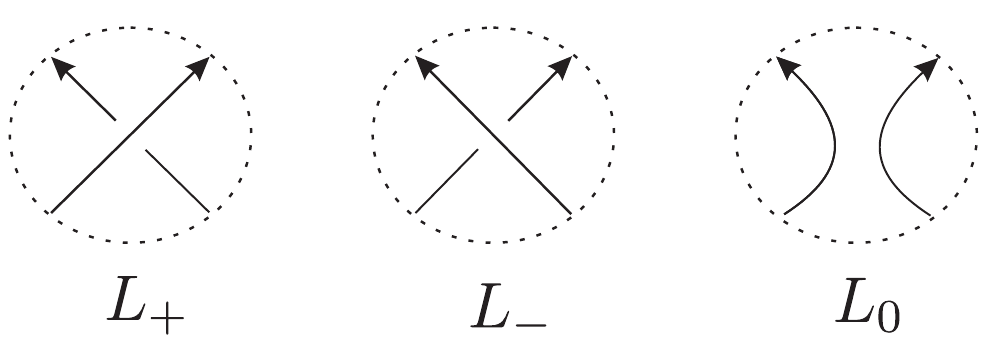}
\caption{\small{Diagrams of links $L_+$, $L_-$ and $L_0$ differ just in a small neighborhood, as shown.}}
\label{D+-0}
\end{figure}

In 1983 Van Buskirk \cite{VanBuskirk} proved that Artin positive links have positive Conway polynomial, that is, all its coefficients are non-negative. Six years later Cromwell \cite{CromwellHom} extended this result to the class of positive links. Recall that Artin-positive links are positive, and these ones are strongly-quasipositive. In this paper we extend their result by proving the following:

\vspace{0.8em}
\noindent \textbf{Theorem \ref{main}}
\emph{Strongly quasipositive links with braid index 3 have positive Conway polynomial.}
\vspace{0.1em}

The plan of the paper is as follows: In Section 2 we recall the definition of the Alexander polynomial in terms of the Burau representation of the braid group. In Section 3 we give explicit formulas for the difference of the Conway polynomial of two links whose associated 3-strands braids differ in an even power of the Garside element; this result leads to a couple of corollaries which are useful in Section 4, which is devoted to prove our main result. The goal of Section 5 is to show that Theorem \ref{main} cannot be generalized to strongly quasipositive links with arbitrary braid index. Finally, in Section 6 we complete some fragments of this paper by giving examples of a link being positive but not Artin-positive, and a link being strongly quasipositive but not positive.

\vspace{0.1cm}

\noindent \textbf{Acknowledgements} \,
I want to thank Pedro M. G. Manchón and Juan González-Meneses for their numerous valuable comments, their suggestions and corrections on preliminary versions of this paper. I am also grateful to Józef H. Przytycki and Maciej Borodzik for fruitful discussions on this problem. \\


\section{Conway polynomial from Burau representation}

Burau \cite{Burau} introduced a linear representation of $\mathbb{B}_n$ by squared matrices of order $n$ over the ring of Laurent polynomials $\mathbb{Z}[t, t^{-1}]$. This representation have been widely studied, being its faithfulness one of the remaining open problems (it is known to be faithful when $n \leq 3$ and unfaithful when $n \geq 5$, but the case $n=4$ is still unsolved).

We will use the reduced Burau presentation $\psi: \mathbb{B}_n \rightarrow GL(n-1, \mathbb{Z}[s,s^{-1}])$ defined by the formula $\psi(\sigma_i) = A_i$, where

$$A_1 = \left( \begin{array}{ccccc}
-s^2 & 0 & 0\\
1 & 1 & 0 \\
0 & 0 & I_{n-3} \end{array} \right), \quad \quad
A_{n-1} = \left( \begin{array}{ccccc}
I_{n-3} & 0 & 0\\
0 & 1 & s^2 \\
0 & 0 & -s^2 \end{array} \right)$$

\noindent and for $1 < i < n-1$

$$
A_i = \left( \begin{array}{ccccc}
I_{i-2} & 0 & 0 & 0 & 0\\
0 & 1 & s^2 & 0 & 0 \\
0 & 0 & -s^2 & 0 & 0 \\
0 & 0 & 1 & 1 & 0 \\
0 & 0 & 0 & 0 & I_{n-i-2}  \end{array} \right)
$$

\noindent with $I_k$ being the identity matrix or order $k$.

With the notation in \cite{KasselTuraev}, $\psi (\beta) = \psi_n^r(\beta)$ after the substitution $t=s^2$. Then Lemma 3.12 and Theorem 3.13 in \cite{KasselTuraev} can be restated to give the following well-known presentation of the Conway polynomial in terms of the reducible Burau representation:

\begin{Teo}\emph{\cite{KasselTuraev}}\label{ConwayBurau}
Let $\alpha \in \mathbb{B}_n$ be a braid and $\widehat{\alpha}$ the link obtained as the closure of $\alpha$. Then the Conway polynomial of $\widehat{\alpha}$ is given by

$$
\nabla(\widehat{\alpha})(z) = (-1)^{n+1} \, \frac{s^{-e_\alpha}}{[n]} \, |\psi(\alpha) - I_{n-1}|
$$

\noindent after the substitution  $s^{-1} - s \, = \, z$, where $[n] = \frac{s^{-n}-s^n}{s^{-1}-s}$ and $e_\alpha \in \mathbb{Z}$ denotes the image of $\alpha$ under the homomorphism $\mathbb{B}_n \, \rightarrow \, \mathbb{Z}$ sending each generator $\sigma_i$ to 1.
\end{Teo}

Note that $e_\alpha$ is well defined, as the exponent sum of a braid word is invariant under the (homogeneous) relations of the braid group. Moreover, $e_\alpha$ is invariant under conjugation.

At this point, we find useful to show a slight modification of a result by P. V. Koseleff and D. Pecker \cite{KosPeck} which simplifies the substitution $s^{-1} - s \, = \, z$ in the formula above by using Fibonacci polynomials. Recall that Fibonacci polynomials are defined by the recurrence relation $F_n(z) = z \, F_{n-1}(z) + F_{n-2}(z)$ for $n \geq 2$, starting with $F_0(z) = 0$ and $F_1(z) = 1$.

There exist closed combinatorial formulas to express the Fibonacci polynomials; however, we will not use them through this paper. For our purposes we need to extend these polynomials to the case when the subindex is negative; we do it in the natural way, by defining $F_{-n}(z) = (-1)^{n+1} F_n(z)$.

Now we are ready to state our adapted version of Lemma 4.1 in \cite{KosPeck}:

\begin{Lema}\emph{\cite{KosPeck}}\label{Koseleff}
Let $F_n(z)$ be the $n^{th}$ Fibonacci polynomial. After the substitution $z = s^{-1} - s$, the identity $(s^{-1})^n + (-s)^n  \, = \, F_{n+1}(z) + F_{n-1}(z)$ holds for any integer $n$.
\end{Lema}


\section{Conway polynomials of 3-braids differing in $\Delta^2$}

From now on, we consider the braid group on $3$ strands, $\mathbb{B}_3$, unless otherwise stated. The following result provides a relation between the Conway polynomials of two closed braids differing in an even power of the Garside element $\Delta = \sigma_1 \sigma_2 \sigma_1$ in $\mathbb{B}_3$.

\begin{Teo}\label{NuestroTeo}
Let $\alpha, \beta \in \mathbb{B}_3$ with $\beta = \Delta^{2k} \,\alpha$, $k > 0$. Then, the difference between the Conway polynomials of their closures is given by
$$
\nabla(\widehat{\beta}) - \nabla(\widehat{\alpha}) = z \displaystyle\sum_{i=0}^{k-1}\left(F_{e_\alpha + 6i + 4} + F_{e_\alpha + 6i + 2}\right),
$$
with $F_n$ being the $n^{th}$  Fibonacci polynomial for any integer $n$.
\end{Teo}
%

\begin{proof}
The reduced Burau representation of $\mathbb{B}_3$ is given by the matrices

$$B_1 = \left( \begin{array}{cc}
-s^2 & 0 \\
1 & 1 \end{array} \right) \quad \quad \mbox{ and } \quad \quad
B_2 = \left( \begin{array}{cc}
1 & s^2 \\
0 & -s^2 \end{array} \right).
$$

Consider the case $k = 1$; the general case will be an extension of the result obtained for this particular case.

As the Garside element in $\mathbb{B}_3$ is $\Delta = \sigma_1\sigma_2\sigma_1$, it follows that $e_\beta = e_\alpha + 6$. Since $\psi(\Delta^2) = s^6 \mbox{Id}_2$, we have that $\mbox{Tr}(\psi(\beta)) = s^{6} \cdot \mbox{Tr}(\psi(\alpha))$. \, We now combine these facts with Theorem $\ref{ConwayBurau}$ for computing the Conway polynomial of both closed braids.

Since for a square matrix $A$ of order two $|A - x \mbox{Id}_2| = x^2 - x \mbox{Tr}(A) + |A|$ and taking into account the substitution $s^{-1} - s \, = \, z$, we have

$$\nabla(\widehat{\alpha}) = \frac{s^{-e_\alpha}}{[3]} |\psi(\alpha) - \mbox{ Id}_2| \, = \, \frac{s^{-e_\alpha}}{[3]} \left[ 1 - \mbox{Tr}(\psi(\alpha)) + |\psi(\alpha)|\right]$$
$$= \frac{1}{[3]} \left[ s^{-e_\alpha} - s^{-e_\alpha} \mbox{Tr}(\psi(\alpha)) + (-1)^{e_\alpha} s^{e_\alpha} \right]. $$

$$\nabla(\widehat{\beta}) = \frac{s^{-e_\beta}}{[3]} |\psi(\beta) - Id| \, = \, \frac{s^{-e_\beta}}{[3]} \left[ 1 - \mbox{Tr}(\psi(\beta)) + |\psi(\beta)|\right]$$
$$ = \frac{1}{[3]} \left[ s^{-e_\alpha -6} - s^{-e_\alpha} \mbox{Tr}(\psi(\alpha)) + (-1)^{e_\alpha} s^{e_\alpha + 6} \right]. $$

Now we compute their difference:

$$\nabla(\widehat{\beta}) - \nabla(\widehat{\alpha}) = \frac{1}{[3]} \left( s^{-e_\alpha - 6} - s^{-e_\alpha} + (-1)^{e_\alpha} (s^{e_\alpha + 6} - s^{e_\alpha}) \right) $$
$$ =  \left[ (-s)^{e_\alpha + 4} + (s^{-1})^{e_\alpha + 4} \right] \, - \, \left[ (-s)^{e_\alpha + 2} + (s^{-1})^{e_\alpha + 2} \right].$$

The second equality comes from the fact that
$$
\left(s^{-1} - s \right) \cdot \left[s^{-e_\alpha - 6} - s^{-e_\alpha} + (-1)^{e_\alpha} (s^{e_\alpha + 6} - s^{e_\alpha}) \right]$$
$$ = \, \left(s^{-3} - s^3 \right) \cdot \left[(-s)^{e_\alpha + 4} + s^{-e_\alpha-4} - (-s)^{e_\alpha + 2} - s^{-e_\alpha - 2} \right].
$$

Applying twice Lemma \ref{Koseleff} we obtain:

$$
\nabla(\widehat{\beta}) - \nabla(\widehat{\alpha}) =  \left(F_{e_\alpha + 5} + F_{e_\alpha + 3} \right) - \left( F_{e_\alpha + 3} + F_{e_\alpha + 1} \right)
$$
$$
= F_{e_\alpha + 5} - F_{e_\alpha + 1} = z \left(F_{e_\alpha + 4} + F_{e_\alpha + 2} \right).
$$

We complete the proof by induction: Suppose the statement true for $1, 2, \ldots, k-1$, and let $\beta = \Delta^{2k} \alpha$, $\gamma = \Delta^{2(k-1)} \alpha$.

$$
\nabla (\widehat{\beta}) - \nabla (\widehat{\alpha}) = \underbrace{ \nabla (\widehat{\beta}) - \nabla (\widehat{\gamma}) }_{ I } + \underbrace{ \nabla (\widehat{\gamma}) - \nabla (\widehat{\alpha}) }_{ II }
$$

$$
= \left[ z \left(F_{e_\gamma + 4} + F_{e_\gamma + 2} \right) \right] + \left[ z \displaystyle\sum_{i=0}^{k-2}\left(F_{e_\alpha + 6i + 4} + F_{e_\alpha + 6i + 2} \right) \right]
$$

$$
= z \displaystyle\sum_{i=0}^{k-1}\left(F_{e_\alpha + 6i + 4} + F_{e_\alpha + 6i + 2} \right).
$$

Applying the induction hypothesis to $I$ and $II$ yields the second equality. The third one holds since $e_\gamma = e_\alpha + 6(k-1)$.

\end{proof}

With some extra work, the theorem above can be deduced from work by Murasugi in \cite[Proposition 4.1]{Murasugi3braids}. He compares the normalized Alexander polynomial of two closed braids differing in an even power of $\Delta$ and provides an expression for their difference, with an indeterminacy on a power of $-t$. We think that the formula we present in Theorem \ref{NuestroTeo} is quite simpler even in the case of the Alexander polynomial (that is, just before the change of variables $s^{-1} - s = z$).

As a consequence of this result, we get an interesting corollary. The result for the even case was proved by Birman in \cite{BirmanJonespol}; as far as we know there is no reference for the odd case.

\begin{Cor}\label{Corolario1}
Let $\alpha \in \mathbb{B}_3$ with $e_\alpha= - 3r \neq 0$, $r > 0$, and consider $\beta = \Delta^{2r} \alpha$:
\begin{itemize}
\item If $r$ is even, then $\nabla(\widehat{\beta}) = \nabla(\widehat{\alpha})$.
\item If $r$ is odd, then $\nabla(\widehat{\beta}) = \nabla(\widehat{\alpha}) \, + 2z \displaystyle\sum_{i=0}^{r-1} F_{-3r + 6i + 4}$.
\end{itemize}
\end{Cor}

\begin{proof}
By applying Theorem \ref{NuestroTeo} we get

\begin{eqnarray*}
\nabla(\widehat{\beta}) & = & \nabla(\widehat{\alpha}) + z \sum_{i=0}^{r-1} \left( F_{-3r + 6i + 4} + F_{-3r + 6i + 2} \right)  \\
 & = & \nabla(\widehat{\alpha}) + z \sum_{i=0}^{r-1} \left( F_{-3r + 6i + 4} + F_{3r - 6i - 4} \right) \\
 & = & \nabla(\widehat{\alpha}) + z \sum_{i=0}^{r-1} \left[ F_{-3r + 6i + 4} + (-1)^{r+1} F_{-3r + 6i + 4} \right] \\
 & = & \nabla(\widehat{\alpha}) + [1 + (-1)^{r+1}] \, z \sum_{i=0}^{r-1} F_{-3r + 6i + 4}.
\end{eqnarray*}

The second equality holds since $F_{-3r+6i+2} = F_{3r-6j-4}$ when $j = (r-1) - i$. The third one holds since $F_{-n} = (-1)^{n+1} F_n$. This completes the proof.

\end{proof}


\section{Strongly quasipositive links with braid index 3}

In this section we continue working with braids on 3 strands. First of all, we want to remark that the class of positive links and the class of strongly quasipositive links are not equal even when considering links with braid index 3. In Section 6 we provide such an example by showing a link with braid index 3 which is not positive but strongly quasipositive.

From the point of view of the presentation given by Artin, $\mathbb{B}_3$ is generated by two generators, $\sigma_1$ and $\sigma_2$. However, if we consider the presentation given by Birman, Ko and Lee, $\mathbb{B}_3$ is generated by three generators: $\sigma_{12} = \sigma_1, \, \sigma_{23} = \sigma_2$ (corresponding to the Artin-generators) and $\sigma_{13} = \sigma_1^{-1} \sigma_2 \sigma_1$. See Figure \ref{Generadores}. From now on, we are going to work with BKL-generators, as strongly quasipositive links are closure of positive braids in terms of these generators.

\begin{figure}
\centering
\includegraphics[width = 11.5cm]{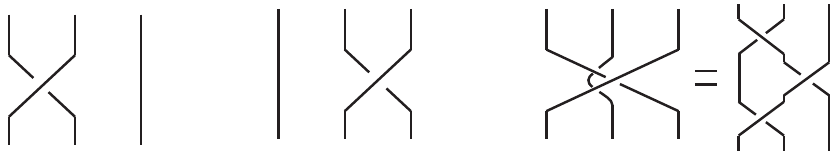}
\caption{\small{BKL-generators $\sigma_{12} = \sigma_1$, $\sigma_{23} = \sigma_2$ and $\sigma_{13} = \sigma_1^{-1}\sigma_2\sigma_1$.}}
\label{Generadores}
\end{figure}

Given an oriented diagram $D$ representing a link $L$, we can construct a resolution tree rooted at $D$ in the following way. Starting from the root, each node would form a triple (parent, leftchild, rightchild) of the form ($D_+, D_-, D_0$). See Figure \ref{trebol}. The edge joining $D_+$ and $D_-$ is labeled with 1, and the one joining $D_+$ and $D_0$ with $z$. This construction codifies the Conway polynomial skein relation. Let $L_1, L_2, \ldots, L_n$ be the leaves in the tree and $P_i,  \, 1 \leq i \leq n$ the product of the labels in the edges on the unique path connecting the leaf $L_i$ and the root of the tree. Then, if we know $\nabla(L_i)$, we can compute the Conway polynomial of the link $L$:
$$
\nabla(L) = \sum_{i=1}^n P_i \cdot \nabla(L_i).
$$

At this point, we can state our main result:

\begin{Teo}\label{main}
Strongly quasipositive links with braid index 3 have positive Conway polynomial.
\end{Teo}

First of all, a result by Stoimenov \cite{Stoimenov3braids} states that any strongly quasipositive link with braid index 3, is the closure of a BKL-positive braid on 3 strands. Hence, we just have to focus on proving the positivity of the Conway polynomial of those closed braids.

The proof of Theorem \ref{main} lies on two results. The first one gives a procedure for constructing a particular resolution tree starting from a BKL-positive braid word, whose branches have positive labels and each leaf is of one of 8 types. In the second one we show that all possible leaves obtained by following the algorithm above, have positive Conway polynomial. The combination of both results completes the proof of Theorem \ref{main}.

\begin{figure}
\centering
\includegraphics[width = 9cm]{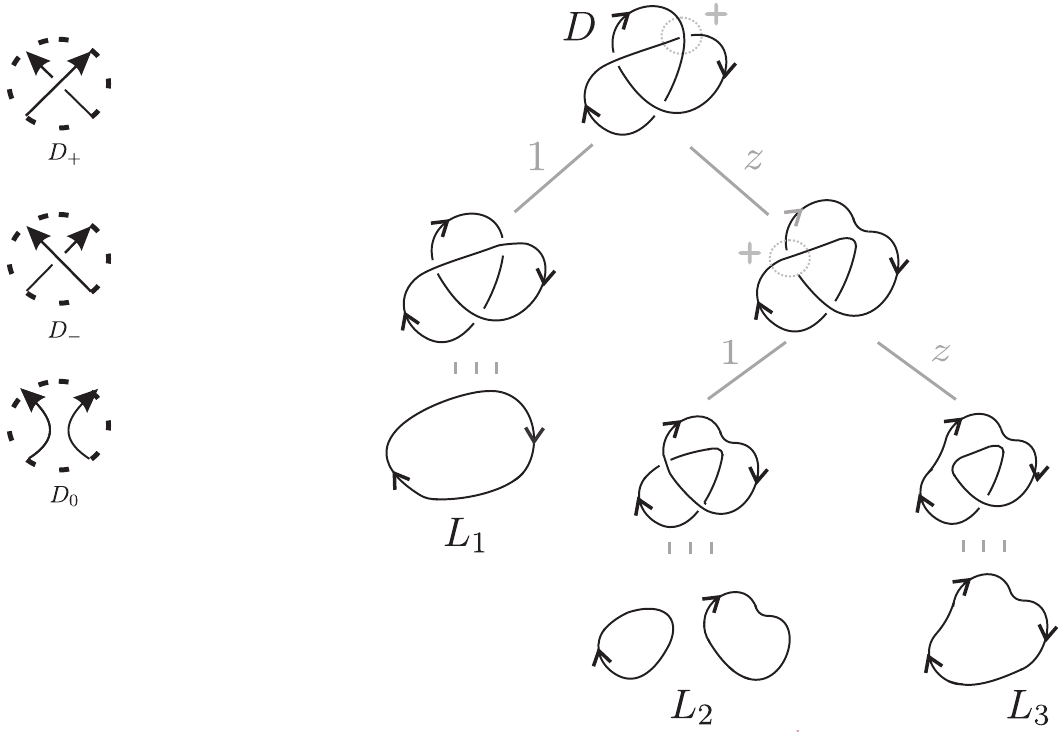}
\caption{\small{A resolution tree for the trefoil knot represented by the diagram $D$. It has three leaves, $L_1$, $L_2$ and $L_3$, and three paths joining them to the root: $P_1 = 1$, $P_2 = z$ and $P_3 = z^2$. As the value of the Conway polynomial of the unknot and split links is $1$ and $0$ respectively, one gets $\nabla(D) = 1 \cdot 1 + z \cdot 0 + z^2 \cdot 1 = 1 + z^2$.}}
\label{trebol}
\end{figure}

\begin{Prop}\label{algoritmo}
Let $\alpha \in \mathbb{B}_3$ be a BKL-positive braid; then, it is possible to construct a resolution tree for the link $\widehat{\alpha}$, whose branches have positive labels and whose leaves are closed braids belonging to the set $\{\varepsilon, \sigma_1, \sigma_1\sigma_2\} \cup \{(\sigma_1\sigma_2\sigma_{13})^k, \, k > 0\},$ with $\varepsilon$ being the trivial braid in $\mathbb{B}_3$.
\end{Prop}

\begin{proof}
Let $w$ be a BKL-positive word representing the braid $\alpha$ and $n$ its length. If $n = 1$, then $w$ is contained in $M_1$, the set containing BKL-positive words of length 1. If $n = 2$, then $w$ is either in $M_2$, the set containing those BKL-positive words of length 2 with two different letters, or it consists on two repeated letters, $w= \sigma_i\sigma_i$, and it can be split into the trivial word, $\varepsilon$, and one word of length one, $\sigma_i$, which is in $M_1$.

Suppose now that $n \geq 3$. If $w = P \sigma_i^2 Q$, with $P$ and $Q$ BKL-positive words, split it by writing $w$ as a node whose left child and right child are $w_1= PQ$ and $w_2= P \sigma_i Q$ respectively; the left branch would be labeled with 1, and the right one with $z$. Note that $w_1$ and $w_2$ have length $n-2$ and $n-1$ respectively. Repeat this procedure with the BKL-positive words obtained each time. Notice that we are not interested in each particular word, but in the link closure of the braid represented by the word (that is, each node in the tree is a BKL-positive word up to conjugation). Hence, we can replace any positive word by another positive word representing the same braid, or by another positive word corresponding to a cyclic permutation of its letters (which corresponds to a conjugate braid, hence to the same link).

Now let us see that every braid which does not belong to $\{\varepsilon\} \cup M_1 \cup M_2 \cup \{(\sigma_1\sigma_2\sigma_{13})^k, \, k > 0\}$ can be split by the above procedure.

Consider a BKL-positive braid word of length at least 3, with no equal consecutive letters. Note that the braids $\sigma_2\sigma_1 = \sigma_{13}\sigma_2 = \sigma_1\sigma_{13}$ are equivalent. Now, start reading the braid word from the left; each time you find either a $\sigma_2\sigma_1$, $\sigma_{13}\sigma_2$ or $\sigma_1\sigma_{13}$ occurrence, write this syllable in such a way that its last letter equals the first letter after it (in cyclic order), so you get two repeated generators together. If no occurrence of $\sigma_2\sigma_1, \sigma_{13}\sigma_2, \sigma_1\sigma_{13}, \sigma_1\sigma_1, \sigma_2\sigma_2, \sigma_{13}\sigma_{13}$ appears in any cyclic permutation of the word, then the letter after every $\sigma_1$ must be $\sigma_2$, the letter after every $\sigma_2$ must be $\sigma_{13}$, and the letter after every $\sigma_{13}$ must be $\sigma_1$, in every cyclic permutation of the word. Therefore, up to a cyclic permutation, the word equals $(\sigma_1\sigma_2\sigma_{13})^k$ for some $k > 0$.

As all the BKL-words in $M_1$ and $M_2$ are conjugated to $\sigma_1$ and $\sigma_1\sigma_2$ respectively, this procedure allows us to construct a resolution tree rooted in $w$, where all the branches have positive labels (either $1$ or $z$), and all the leaves belong to the set $\{\varepsilon, \sigma_1, \sigma_1\sigma_2\} \cup \{(\sigma_1\sigma_2\sigma_{13})^k, k>0\}$.
\end{proof}

\begin{figure}
\centering
\includegraphics[width = 13.5cm]{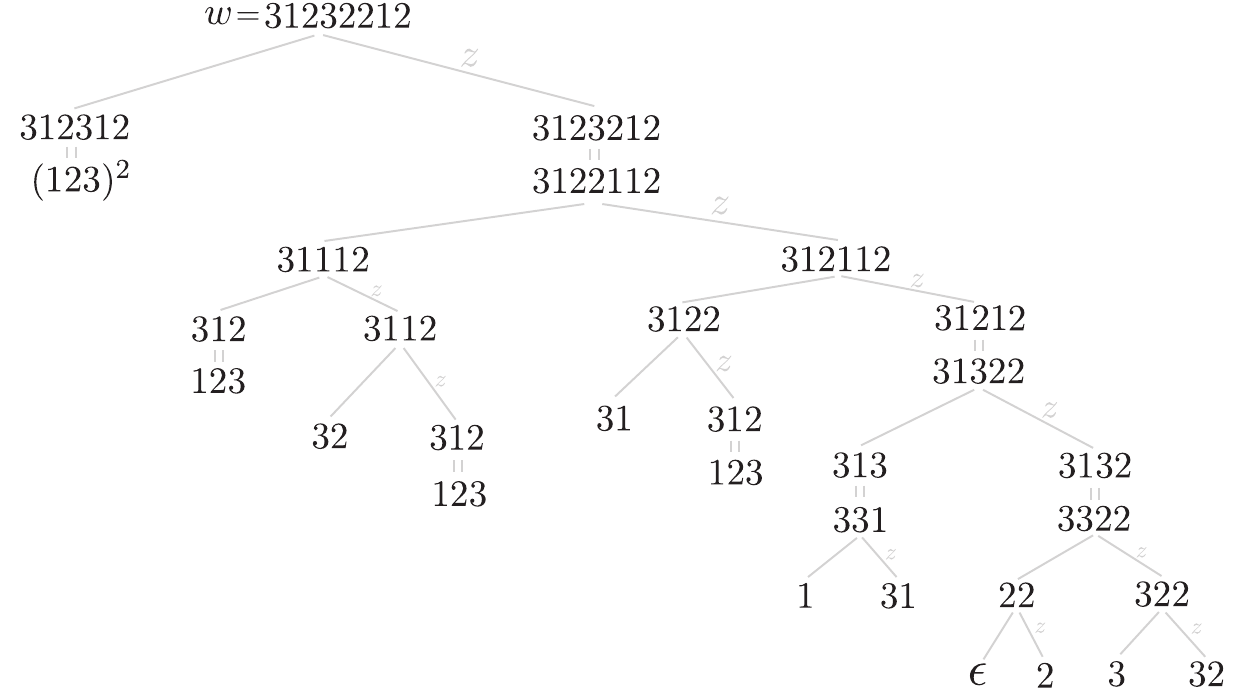}
\caption{\small{This image illustrates the algorithm in the proof of Proposition \ref{algoritmo}. The sign ``='' represents a cyclic move in the order of the letters (that is, a conjugation). Letters $1$, $2$ and $3$ represent generators $\sigma_1$, $\sigma_2$ and $\sigma_{13}$, respectively. }}
\label{Palabra}
\end{figure}

At this point, we just need to show that the closure of the braids in the above set have positive Conway polynomial. As closing the trivial braid or $\sigma_1$ gives an split link, their Conway polynomials are null. The closure of a braid represented by a word with two different letters, lets say $\sigma_1\sigma_2$, is the trivial knot, so its Conway polynomial is $1$.

It remains to prove the case of links which are closure of braids of the form $(\sigma_1\sigma_2\sigma_{13})^k$, with $k>0$. These are non-split links with 2 or 3 components, depending on the parity of $k$. Computing their Conway polynomial is not trivial; as a particular case of Corollary \ref{Corolario1}, we obtain the following result (the case $k$ even was also computed by Stoimenov in \cite{Stoimenov3braids}):

\begin{Cor} \label{Cor123}
The closure of the braid on 3 strands $(\sigma_1\sigma_2\sigma_{13})^k$ has positive Conway polynomial, for any integer $k > 0$. In fact, $\nabla(\widehat{(\sigma_1\sigma_2\sigma_{13})^k}) = 2z \displaystyle\sum_{i=0}^{k-1} F_{-3k + 6i - 4}$ when $k$ is odd, and it is null when $k$ is even.
\end{Cor}

\begin{proof}

It is easy to check that $\sigma_1\sigma_2\sigma_{13} = \Delta^2 (\sigma_2^{-1})^3$, so $(\sigma_1\sigma_2\sigma_{13})^k = \Delta^{2k} (\sigma_2^{-1})^{3k}$, since $\Delta^2$ is central. Now apply Corollary \ref{Corolario1}, with $\alpha = (\sigma_2^{-1})^{3k}$ (hence $e_\alpha = -3k$) and $\beta = \Delta^{2k} (\sigma_2^{-1})^{3k} = (\sigma_1\sigma_2\sigma_{13})^k$.

If $k$ is even, then $\nabla((\widehat{\sigma_1\sigma_2\sigma_{13}})^{k}) = \nabla(\widehat{(\sigma_2^{-1})^{3k}}) = 0$, since the closure of $(\sigma_2^{-1})^{3k}$ in $\mathbb{B}_3$ is a split link.

If $k$ is odd, then $\nabla((\widehat{\sigma_1\sigma_2\sigma_{13}})^{k}) = 2z \sum_{i=0}^{k-1} F_{-3k + 6i + 4}$. As the Fibonacci polynomials in the summation have odd subindices, all their coefficients are positive.
\end{proof}

This completes the proof of Theorem \ref{main}.


\section{Trying to extend the result}

In this section we consider the problem of extending the previous result to a higher number of strands, that is, we study whether every strongly quasipositive link has positive Conway polynomial. The following result gives a negative answer to this question by showing a counterexample:

\begin{Prop}\label{counterex}
There are strongly quasipositive links having non-positive Conway polynomial.
\end{Prop}
\begin{proof}
Consider the BKL-positive braid on 6 strands $\alpha = \sigma_{16}\sigma_{16}\sigma_{46}\sigma_{35}\sigma_{24}\sigma_{13}\sigma_{25}$. Its closure is a strongly quasipositive link, whose Conway polynomial is $\nabla(\widehat{\alpha}) = -z^2 + 1$, which is non-positive. (We computed $\nabla(\widehat{\alpha})$ by using a C++ version of the program br9z.p, developed by Short and Morton in 1985: \, \texttt{http://www.liv.ac.uk/~$\sim$~su14/knotprogs.html}).
\end{proof}

In \cite{deltamove}, in the proof of Corollary 88 Rudolph stated that every Seifert matrix of a given link can be obtained as the Seifert matrix of a strongly quasipositive link. As a consequence, given a link $L$ with Conway polynomial $\nabla(L)$, there would exist a strongly quasipositive link $L'$ having the same Conway polynomial, that is, verifying $\nabla(L) = \nabla(L')$. He also gives a procedure for constructing $L'$ as a closed BKL-positive braid, starting from a braid diagram of $L$. This result would provide an infinite family of examples of strongly quasipositive links having non-positive Conway polynomial. However, we think that there is a problem with the proof of this result: it is claimed that the procedure for obtaining $L'$ starting from $L$ (a sequence of \emph{doubled-delta moves}, also called \emph{trefoil insertion}) preserves the Seifert matrix. After applying this move to the braid $\beta = \sigma_1\sigma_1\sigma_1^{-1} \in \mathbb{B}_2$, one obtains $\beta' = \sigma_{16}\sigma_{16}\sigma_{25}\sigma_{13}\sigma_{24}\sigma_{35}\sigma_{46}$. The closure of $\beta$ is the trivial knot, hence $\nabla(\widehat{\beta}) = 1$; however $\nabla(\widehat{\beta'}) = 7z^2 + 1$. This contradicts the fact that doubled-delta moves preserve the Seifert matrix.


\section{Appendix}

In this section we show that the families of Artin positive, positive and strongly quasipositive links are not equivalent, even in the case of links with braid index 3. We give two examples: the first one is a positive knot which is not Artin-positive, and the second one a strongly quasipositive link which is not positive.

\begin{Prop}\label{homnobraidhom}
There are positive links which are not Artin-positive.
\end{Prop}

\begin{proof}
In Figure \ref{posArtinBKL} it is shown a positive diagram of the knot $5_2$; its braid index is 3, as it is the closure of the braid $\alpha = \sigma_2^{-3}\sigma_1^{-1}\sigma_2\sigma_1^{-1}$. Suppose now that $5_2$ is the closure of a braid $\gamma$ represented by a positive Artin-word $w$. We take $w$ with minimal length. Let $D$ be the associated positive diagram. As projection surfaces constructed from homogeneous diagrams have minimal genus \cite{CromwellHom} and positive diagrams are homogeneous, $g(S_D) = g(5_2) =1$ leads to $s + 1 = c$, where $s$ and $c$ are the number of Seifert discs and bands in $S_D$, the projection surface arising from $D$. Notice that $s$ is the number of strands and $c$ is the number of crossings of $\gamma$.

Since $c \geq 5$, $\gamma$ must have at least 4 strands, and then some generator $\sigma_i$ must appear at most once. All generators must appear, since $5_2$ is a knot (a one-component link), so there exists one generator appearing exactly once, and this is a nugatory crossing. This is a contradiction with the minimality of $w$ since $5_2$ is prime.
\end{proof}

\begin{Prop}
Links in the family $\{(\widehat{\sigma_1\sigma_2\sigma_{13}})^{k}, \, \, k \mbox{ even}\}$ are non-positive but strongly quasipositive.
\end{Prop}

\begin{proof}
From the definition, it is immediate to check that links of the form $(\widehat{\sigma_1\sigma_2\sigma_{13}})^{k}$ are strongly quasipositive. The proof of their non-positivity when $k$ is even lies in a couple of additional results.

The first one is a result by Baader \cite{Baader} which states that a strongly quasipositive link is positive if and only if it is homogeneous. Peter Cromwell proved in \cite{CromwellHom} that, given an homogeneous link $L$ of $\mu(L)$ components, its Conway polynomial is related with its genus by the formula $2 g(L) = \mbox{maxdeg }(\nabla(L)) -\mu(L) + 1$.

Let $L$ be a link in the family $(\widehat{\sigma_1\sigma_2\sigma_{13}})^{k}$, and suppose that $L$ is positive. Since links in this family have 3 components, as a consequence of the results above $L$ should verify $2 g(L) = \mbox{maxdeg }(\nabla(L)) - 2$.

Let $S$ be the quasipositive surface associated to the word in the form $(\sigma_1\sigma_2\sigma_{13})^k$ representing $L$ (see Figure \ref{quasipossurf}). The Euler characteristic of this kind of surfaces can be computed as $\chi(S) = d_S - b_S$, with $b_S$ and $d_S$ being the number of bands and discs in $S$ respectively; as a consequence, $2 g(S) = 2 - \mu(L) + b_S - d_S$.

Rudolph proved in \cite{Rudolph1} that quasipositive surfaces have minimal genus for the link they are spanning. Hence, $L$ should verify $\mbox{maxdeg }(\nabla(L)) = 1 + b_S - d_S$. As we are working in $\mathbb{B}_3$, there are 3 discs in $S$, so $\mbox{maxdeg }(\nabla(L)) + 2 = b_S$.

Hence, if $L$ were positive the number of bands in $S$ and its Conway polynomial should be related in the way above. We proved in Corollary \ref{Cor123} that $\nabla(L) = 0$. However, the number of bands in $S$ equals $3k$, yielding a contradiction.

\end{proof}

\vspace{0.2cm}

\noindent \textbf{Funding:} This study was partially funded by de Spanish Ministry of Economy and Competitiveness and FEDER (Projects MTM2010-19355 and MTM2013-44233-P), and by the Regional Government of Andalousia (Project P09-FQM-5112). \\

\noindent \textbf{Conflict of Interest:} The author declares that she has no conflict of interest.

\bibliographystyle{plain}
\bibliography{Bibliograf}

\begin{thebibliography}{10}

\bibitem{Artin1}
E.~Artin.
\newblock Theorie der {Z}öpfe.
\newblock {\em Abh. Math. Sem. Hamburg}, 4:47--72, 1925.

\bibitem{Artin2}
E.~Artin.
\newblock Theory of {B}raids.
\newblock {\em Ann. of Math.}, 48:101--126, January, 1947.

\bibitem{Baader}
S.~Baader.
\newblock Quasipositivity and homogeneity.
\newblock {\em Math. Proc. Camb. Phil. Soc.}, 139:287--290, 2005.

\bibitem{BirmanJonespol}
J.~Birman.
\newblock On the jones polynomial of closed 3-braids.
\newblock {\em Inventiones mathematicae}, 81:287--294, 1985.

\bibitem{BirmanKoLee}
J.~Birman, K.H. Ko, and S.J. Lee.
\newblock A new approach to the {W}ord and {C}onjugacy {P}roblems in the
  {B}raid {G}roups.
\newblock {\em Advances in Math.}, 139:322--353, 1998.

\bibitem{Burau}
W.~Burau.
\newblock Über zopfgruppen und gleichsinnig verdrillte verkettungen.
\newblock {\em Abhandlungen aus dem Mathematischen Seminar der Universität
  Hamburg}, 11:179--186, 1935.

\bibitem{CromwellHom}
P.~Cromwell.
\newblock Homogeneous links.
\newblock {\em J. Lond. Math. Soc.}, 2(39):535--552, 1989.

\bibitem{KasselTuraev}
C.~Kassel and V.~Turaev.
\newblock {\em Braid groups}, volume 247.
\newblock Graduate Texts in Mathematics - Springer, 2008.

\bibitem{KosPeck}
P.V. Koseleff and D.~Pecker.
\newblock On alexander-conway polynomials of two-bridge links.
\newblock {\em Journal of Symbolic Computation}, 68:215--229, 2015.

\bibitem{Murasugi3braids}
K.~Murasugi.
\newblock {\em On closed 3-braids}, volume 151.
\newblock Memoirs of the American Mathematical Society, 1974.

\bibitem{RudolphStronglyquasipos}
L.~Rudolph.
\newblock A congruence between link polynomials.
\newblock {\em Mathematical Proceedings of the Cambridge Philosophical
  Society}, 107:319--327, 1990.

\bibitem{deltamove}
L.~Rudolph.
\newblock Knot theory of complex plane curves.
\newblock {\em Handbook of Knot Theory}, pages 349--427, 2005.

\bibitem{Rudolph1}
Lee Rudolph.
\newblock Quasipositive plumbing (constructions of quasipositive knots and
  links, {V}).
\newblock {\em Proc. Amer. Math. Soc.}, 126(1):257--267, 1998.

\bibitem{Rudolph2}
Lee Rudolph.
\newblock Positive links are strongly quasipositive.
\newblock {\em Geometry and Topology Monographs: Proceedings of the Kirbyfest},
  2:555--562, 1999.

\bibitem{MarithaConj}
M.~Silvero.
\newblock On a conjecture by {K}auffman on alternative and pseudoalternating
  links.
\newblock {\em http://arxiv.org/abs/1402.4599}, 2014.

\bibitem{Stoimenov3braids}
A.~Stoimenov.
\newblock Properties of closed 3-braids.
\newblock {\em http://arxiv.org/abs/math/0606435}, 2007.

\bibitem{VanBuskirk}
J.M. Van~Buskirk.
\newblock Positive knots have positive {C}onway polynomials.
\newblock {\em Knot Theory and Manifolds - Lecture Notes in Mathematics},
  1144:146--159, 1983.

\end{thebibliography}

\end{document}